\documentclass[11pt,twoside]{amsart}

\usepackage{amsmath}
\usepackage{amsthm}
\usepackage{amsfonts, amssymb}
\usepackage{mathrsfs}
\usepackage[all]{xy}
\usepackage{url}
\usepackage{graphics}
\usepackage{epstopdf}

\usepackage{latexsym}
\usepackage[dvips]{graphicx}
\def\VR{\kern-\arraycolsep\rule[-.3\baselineskip]{0pt}{\baselineskip}\vrule &\kern-\arraycolsep}
\def\vr{\kern-\arraycolsep & \kern-\arraycolsep}

\theoremstyle{plain}
\newtheorem{lema}{Lemma}[section]
\newtheorem{prop}[lema]{Proposition}
\newtheorem{teo}[lema]{Theorem}
\newtheorem{conj}[lema]{Conjecture}

\theoremstyle{remark}

\newtheorem{obs}[lema]{Remark}

\theoremstyle{definition}
\newtheorem{defi}[lema]{Definition}
\newtheorem{ej}[lema]{Example}

\newcommand{\p}{\mathcal{P}}

\newcommand{\len}{\textrm{l}}
\newcommand{\occ}{\textrm{occ}}

\def\Z{\mathbb{Z}}
\def\R{\mathbb{R}}

\begin{document}

\title[A new test for asphericity and diagrammatic reducibility]{A new test for asphericity and diagrammatic reducibility of group presentations}

\author[J.A. Barmak]{Jonathan Ariel Barmak}
\author[E.G. Minian]{Elias Gabriel Minian}

\address{Departamento  de Matem\'atica--IMAS\\
 FCEyN, Universidad de Buenos Aires\\ Buenos
Aires, Argentina}

\email{jbarmak@dm.uba.ar}
\email{gminian@dm.uba.ar}

 \thanks{Researchers of CONICET. Partially supported by grants ANPCyT PICT-2011-0812, CONICET PIP 112-201101-00746 and UBACyT 20020130100369.}

\begin{abstract}
We present a new test for studying asphericity and diagrammatic reducibility of group presentations. Our test can be applied to prove diagrammatic reducibility in cases where the classical weight test fails. We use this criterion to generalize results of J. Howie and S.M. Gersten on asphericity of LOTs and of Adian presentations, and derive new results on solvability of equations over groups.  We also use our methods to investigate a conjecture of S.V. Ivanov related to Kaplansky's problem on zero divisors: we strengthen Ivanov's result for locally indicable groups and prove a weak version of the conjecture.
\end{abstract}



\maketitle

\section{Introduction}

The notion of asphericity is central to topology, geometry and algebra. Recall that a path-connected space $X$ is called aspherical if its homotopy groups $\pi_n(X)$ are trivial for  $n\geq 2$. A longstanding open problem in low dimensional topology is the Whitehead asphericity conjecture, which  asserts that a path-connected subcomplex of an aspherical $2$-complex is itself aspherical \cite{bo,Si2,wh}. There have been considerable advances in this question which include works of Cockcroft \cite{co}, Adams \cite{ad} and Howie \cite{Ho2,Ho3}. A closely related unsolved problem is whether ribbon disc complements are necessarily aspherical \cite{Ho4,Ho5}. This question was shown to be equivalent to whether complexes associated to \emph{labeled oriented trees} (LOTs) are aspherical. LOT-complexes are combinatorially encoded $2$-complexes which naturally appear as spines of ribbon disc complements \cite{Ho4}. A proof of the asphericity of ribbon disc complements would generalize, in some sense, Papakyriakopoulos' analogous result for knot complements. It is easy to see that every LOT-complex embeds in a contractible $2$-complex. Therefore LOT-complexes are considered test cases for the Whitehead conjecture. The concept of asphericity plays a key role in combinatorial group theory (see for example \cite{cch, Ger, Ger4, Ho2, Ho4, Iva, ls}). In this context, a group presentation is called aspherical if its standard $2$-complex is aspherical. In geometry, asphericity is fundamental in the study of manifolds and it is related to the theory of hyperbolic groups (see \cite{blw,gr,lu}).

The concept of \emph{diagrammatic reducibility}, first studied by Sieradski \cite{Si1} and Gersten \cite{Ger,Ger2}, is related to asphericity. Any diagrammatically reducible (DR) presentation is aspherical. On the other hand diagrammatic reducibility is intimately connected with the problem of solving equations over groups. The Kervaire-Laudenbach-Howie conjecture asserts that any independent system of equations over a group $H$ has a solution in an overgroup of $H$ \cite{Ho1,Ger2,Rou}. Howie showed that the conjecture is true for locally indicable groups. Gersten proved that if a presentation $\p$ is DR then all systems of equations over an arbitrary group $H$ modeled on $\p$  are solvable in an overgroup of $H$ \cite{Ger2} (see also \cite{br,br2,kr} for more results in these directions). Gersten also used the concept of diagrammatic reducibility to study subgroups of hyperbolic groups \cite{Ger3}.

One of the best-known tests for asphericity and diagrammatic reducibility of $2$-complexes (or group presentations) is Gersten's weight test \cite{Ger,Ger2} which is a generalization of Sieradski's coloring test \cite{Si1}. Gersten's weight test was further developed by Pride \cite{pr} and generalized by Huck and Rosebrock \cite{HR}. All these variants of the weight test are based on the combinatorial Gauss-Bonnet theorem.
Howie attacked asphericity problems with a different approach, using the notion of local indicability. Recall that a group $G$ is \emph{indicable} if it admits a nontrivial homomorphism to the infinite cyclic group, and it is \emph{locally indicable} if all its nontrivial finitely generated subgroups are indicable. Howie showed that a connected $2$-complex with locally indicable fundamental group and trivial second homology group is aspherical \cite{Ho2}. He also used the concept of local indicability to deduce that \emph{reducible} presentations with no proper powers are aspherical and to prove asphericity of certain classes of LOT-complexes \cite{Ho4}. In \cite{Ho5} he showed that for $n\geq 3$, ribbon $n$-knots in $S^{n+2}$ admit minimal Seifert manifolds provided their associated LOTs have diameter at most $3$.

In this paper we present a new test for studying asphericity and diagrammatic reducibility  of group presentations. The I-test provides a criterion for deciding when a presentation $\p$ of an indicable group $G$ is DR. Our test is based on a description of the second homotopy group of the $2$-complex $K_{\p}$ associated to $\p$ as the kernel of the boundary map $C_2(\widetilde{K}_{\p})\to C_1(\widetilde{K}_{\p})$ in the cellular chain complex of the universal cover of $K_{\p}$, and it uses basic linear algebra methods. We compare our test with the different variations of the weight test and use our methods to prove asphericity and diagrammatic reducibility in cases where known tests fail.

We use our test to obtain generalizations of results of Gersten and Howie on Adian presentations and LOTs \cite{Ger,Ho4}. We introduce the notions of \emph{deforestable} and \emph{weakly deforestable} labeled oriented graph (LOG) and show that their associated presentations are diagrammatically reducible. These classes strictly contain those LOGs $\Gamma$ such that either the initial graph $I(\Gamma)$ or the terminal graph $T(\Gamma)$ has no cycles (cf. \cite{Ho4}).

We also apply our methods to investigate a conjecture of Ivanov related to Kaplansky problem on zero divisors \cite{Iva}. Suppose $\p$ is a finite presentation and $\mathcal{Q}$ is obtained from $\p$ by adding one generator $x$ and one relator $r$. Ivanov conjectured that, under certain hypotheses, the asphericity of $\p$ implies that of $\mathcal{Q}$. Ivanov showed that a counterexample to this conjecture would provide a torsion-free group $G$ whose integral group ring $\Z G$ contains zero divisors. He proved that the conjecture holds in certain cases, including the case where the group $H$ presented by $\p$ is locally indicable. This follows from results of Howie. In fact under the hypotheses of the conjecture it is immediate that the locally indicability of $H$ implies that of $G$. We prove that the conjecture remains valid if one only requires that some specific (finitely many) subgroups of $G$, determined by the relator $r$, be indicable (see Theorem \ref{ivanovli} below). In Theorem \ref{ivanovp} we show that the conjecture is true if we are allowed to \emph{perturb} the relation $r$.

In the last section of the article we apply our test to derive new results on solvability of equations over groups. We concentrate on the existence of solutions of one equation with many variables. Our results in this direction provide a partial complement to an old result of Brick \cite{br} and recent results of Klyachko and Thom \cite{kt}.

\section{The I-test} \label{seccionppal}

Let $\p=\langle x_1, x_2, \ldots, x_n | r_1, r_2, \ldots ,r_m\rangle$ be a presentation of a group $G$, where $\mathcal{A}=\{x_1,x_2, \ldots, x_n\}$ is an alphabet and the relators $r_j$ are non-necessarily reduced words in $\mathcal{A}\cup \mathcal{A}^{-1}$. If $w$ is a word in $\mathcal{A}\cup \mathcal{A}^{-1}$, $w$ will also denote the corresponding element in the free group $F(x_1,x_2,\ldots , x_n)$ with basis $\mathcal{A}$ and the element $p(w)\in G$, where $p:F(x_1,x_2,\ldots ,x_n)\to G$ is the quotient map. In particular if $g\in G$, $gw$ denotes the element $g.p(w)\in G$. The second homotopy group $\pi_2(K_\p)$ of the associated $2$-complex is isomorphic to the second homology group $H_2(\widetilde{K}_\p)$ of the universal cover of $K_{\p}$. The complex $\widetilde{K}_\p$ has one $0$-cell $e^0_g$ for each element $g\in G$. For each $1\le i\le n$, there is an oriented $1$-cell $e^1_{i,g}$ from $e^0_g$ to $e^0_{gx_i}$. For each $g\in G$ and $1\le j\le m$ there is an oriented $2$-cell $e^2_{j,g}$. If $x_i$ is a letter of $r_j$, say $r_j=wx_iw'$ for certain words $w,w'$, then $e^1_{i,gw}$ is a face of $e^2_{j,g}$ with incidence $1$. If $r_j=wx_i^{-1}w'$, then $e^1_{i,gwx_i^{-1}}$ is a face of $e^2_{j,g}$ with incidence $-1$. This yields a description of $\pi_2 (K_\p)$ that goes back to Reidemeister and Whitehead (see \cite[pp. 84]{Si2}), as the kernel of the boundary map $\partial :C_2(\widetilde{K}_\p)\to C_1(\widetilde{K}_\p)$ in the cellular chain complex. Theorem \ref{ecuaciones} below summarizes these ideas giving equations for $\pi_2$ which can be read off from the presentation. 
If $w=x_{i_1}^{\epsilon_1}x_{i_2}^{\epsilon_2} \ldots x_{i_l}^{\epsilon_l}$ is a (non-necessarily reduced) word in $\mathcal{A}\cup \mathcal{A}^{-1}$, where $\epsilon_{k,w}=\epsilon_k=\pm 1$ for each $1\le k \le l$, the length $\len(w)=l$ of $w$ is the number of letters in $w$. Given $1\le k\le \len(w)$, denote by $w^{(k)}$ the subword of $w$ obtained by removing the first $k-1$ letters of $w$. That is, $w^{(k)}=x_{i_k}^{\epsilon_k}x_{i_{k+1}}^{\epsilon_{k+1}}\ldots x_{i_l}^{\epsilon_{l}}$. For $1\le k\le \len(w)$ define $s(k,w)=w^{(k)}$ if $\epsilon_k=1$, and $s(k,w)=w^{(k+1)}$ if $\epsilon_k=-1$. Let $1\le i\le n$. Denote by $\occ(x_i,w)=\{1\le k\le l \ | \ x_{i_k}=x_i\}$ the set of positions in which the letter $x_i$ (or its inverse $x_i^{-1}$) occurs in $w$.

\begin{teo} \label{ecuaciones}
Let $\p=\langle x_1, x_2, \ldots , x_n \ | \ r_1, r_2, \ldots, r_m \rangle$ be a presentation of a group $G$. For each $1\le i\le n$ and each $g\in G$ consider the integral linear equation $E_{i,g}$: $$\sum\limits_{1\le j\le m} \sum\limits_{k\in \occ(x_i, r_j)} \epsilon_{k,r_j}n^j_{gs(k,r_j)}=0 $$ in the unknowns $n^j_h$, for $1\le j\le m$, $h\in G$.

Then $K_{\p}$ is aspherical if and only if the unique solution of the system $\{E_{i,g}\}_{i,g}$ with finitely many nontrivial $n^j_g$ is the trivial solution, $n^j_g=0$ for each $1\le j\le m, g\in G$.  
\end{teo}

\begin{obs} \label{terminos}
A term $n^j_h$ appears in the equation $E_{i,g}$ if and only if $e^1_{i,g}$ is a face of $e^2_{j,h}$.
\end{obs}

\begin{ej} \label{ej1}
Let $\mathcal{P}=\langle x,y,z,w | x^2y^2z^2, xyx^{-1}zyz^{-1}, w^2x^{-1}w^{-1}z\rangle$. The equations associated to the generators $x,y,z,w$ are respectively 
\medskip

$n^1_{gx^2y^2z^2}+n^1_{gxy^2z^2}+n^2_{gxyx^{-1}zyz^{-1}}-n^2_{gzyz^{-1}}-n^3_{gw^{-1}z}=0,$

$n^1_{gy^2z^2}+n^1_{gyz^2}+n^2_{gyx^{-1}zyz^{-1}}+n^2_{gyz^{-1}}=0,$

$n^1_{gz^2}+n^1_{gz}+n^2_{gzyz^{-1}}-n^2_g+n^3_{gz}=0,$

$n^3_{gw^2x^{-1}w^{-1}z}+n^3_{gwx^{-1}w^{-1}z}-n^3_{gz}=0.$
\end{ej}

Let $G$ be a finitely presented group. Recall that $G$ is said to be indicable if there exists an epimorphism from $G$ onto the infinite cyclic group $\mathbb{Z}$. We denote by $G^{ab}=G/[G,G]$ the abelianization of $G$. Let $\mathcal{P}=\langle x_1, x_2, \ldots , x_n | r_1, r_2, \ldots, r_m \rangle$ be a presentation of $G$. Denote by $q:F(x_1, x_2, \ldots , x_n)\to F(x_1,x_2,\ldots ,x_n)^{ab}=\Z^n$ the quotient map. We identify $q(x_i)\in F(x_1,x_2,\ldots ,x_n)^{ab}$ with the $i$-th vector of the standard basis of $\Z^n$.

\begin{lema}
$G$ is indicable if and only if there is a nonzero vector $v\in \R^n$ orthogonal to each $q(r_j)$.
\end{lema}
\begin{proof}
An epimorphism $G\to \Z$ induces an epimorphism $G^{ab}\to \Z$. Since $G^{ab}$ is isomorphic to $\Z^n/\langle q(r_j)\rangle_j$, there exists an epimorphism $\Z^n\to \Z$ which is zero in each $q(r_j)$. A homomorphism $\Z^n \to \Z$ is the multiplication by a vector $v\in \Z^n$. Conversely, a nonzero vector $v\in \R^n$ orthogonal to each $q(r_j)$ induces a nontrivial homomorphism $f:G^{ab}\to \R$. Since $f(G^{ab})$ is a nontrivial finitely generated torsion free abelian group, it is isomorphic to $\Z^r$ for some $r\ge 1$, and then $G$ is indicable.   
\end{proof}

Assume that $G$ is indicable and let $v\in \R^n$ be a vector orthogonal to each $q(r_j)$. Given $w\in F(x_1, x_2, \ldots, x_n)$, its \textit{weight} (relative to $v$) is $\overline{w}=\langle q(w), v\rangle \in \R$, where $\langle,\rangle$ denotes the standard inner product of $\R^n$. If $g\in G$ is represented by an element $w\in F(x_1, x_2, \ldots, x_m)$, we define its weight as $\overline{g}=\overline{w}$. Note that the definition does not depend on the representative. We define the \textit{weight matrix} $M$ of $\p$ as follows. It is an $n\times m$ matrix whose entries are families of real numbers. The family $M_{i,j}$ is

$$M_{i,j}=\{\overline{s(k,r_j)}\}_{k\in \occ(x_i,r_j)}$$

In other words, $M_{i,j}$ contains the weights of the subindices $g$ of the unknowns $n^j_g$ which appear in the equation $E_{i,1}$ (corresponding to the generator $x_i$ and the trivial element $1\in G$) in the statement of Theorem \ref{ecuaciones}. Of course, $M$ depends on the vector $v$. It is important to note that a number can appear in $M_{i,j}$ with multiplicity and that $M_{i,j}$ is empty if $x_i$ does not occur in $r_j$.

\begin{defi} \label{defiasph}
An $n\times m$ matrix $M$ whose entries are (possibly empty) families of real numbers will be called \textit{good} if $n\ge m$ and there exists an ordering $j_1, j_2, \ldots , j_m$ of the columns of $M$ and an ordering $i_1, i_2, \ldots , i_m$ of $m$ of the rows, such that for each $1\le k\le m$

\begin{enumerate}
\item $M_{i_k,j_k}$ is non-empty,
\item the maximum $\lambda_k$ of $M_{i_k,j_k}$ is the maximum of the whole row $\bigcup\limits_{j=1}^m M_{i_k,j}$ and
\item the multiplicity of $\lambda_k$ in $\bigcup\limits_{l=k}^m M_{i_k,j_l}$ is one.
\end{enumerate}

A presentation $\mathcal{P}=\langle x_1, x_2, \ldots ,x_n |r_1, r_2, \ldots, r_m \rangle$ satisfies the \textit{I-test} if there exists a vector $v\in \R^n$, orthogonal to each $q(r_j)$ such that the corresponding weight matrix is good. In this case we say that $\p$ satisfies the I-test for $v$, and for the orderings $j_1, j_2, \ldots , j_m$ and $i_1, i_2, \ldots , i_m$ that make $M$ good.
\end{defi}

In Example \ref{ej1} above, the vector $v=(1,0,-1,2)$ is orthogonal to $q(r_1)=(2,2,2,0)$, $q(r_2)=(0,2,0,0)$ and $q(r_3)=(-1,0,1,1)$. The weight matrix $M$ is

\begin{displaymath}\bordermatrix{&\phantom{-}r_1& & \phantom{-}r_2 & & \phantom{-}r_3 \cr 
        x& \phantom{-}0,-1 &  & \phantom{-}0,0 & & \phantom{0,-1,} -3  \cr
        y& -2,-2 &  & -1,1  & & \phantom{0,-2,-} \emptyset \cr 
        z& -2,-1 &  & \phantom{-}0,0 & & \phantom{0,-2,} -1 \cr
        w& \phantom{-2,-}\emptyset & & \phantom{-0,} \emptyset & & 0,-2,-1}
\end{displaymath}

$\mathcal{P}$ satisfies the I-test with respect to the orders $2,3,1$ of the columns and $2,4,1$ of the rows: the maximum of the second row is $1\in M_{2,2}$ and has multiplicity one in $M_{2,2}\cup M_{2,3}\cup M_{2,1}=\{-2,-2,-1,1\}$, the maximum of the fourth row is $0\in M_{4,3}$ and has multiplicity one in $M_{4,3}\cup M_{4,1}=\{-2,-1,0\}$, the maximum of the first row is $0\in M_{1,1}$ and has multiplicity one in $M_{1,1}=\{-1,0\}$.

\begin{teo} \label{main}
If $\mathcal{P}$ satisfies the I-test, it is aspherical.
\end{teo}
\begin{proof}
Suppose $\mathcal{P}$ satisfies the I-test and let $v$, $j_1, j_2, \ldots , j_m$ and $i_1, i_2, \ldots, i_m$ be as in the definition. If $\p$ is not aspherical, there is a nontrivial family of integer numbers $\{n^j_g\}_{j,g}$ with finite support, which is a solution to the equations $E_{i,g}$. Let $\alpha=\min \{ \overline{h} \ | \ h\in G$ and there exists $1\le j\le m$ such that $n^j_h\neq 0 \}$. Let $t=\min \{1\le l\le m \ | \ $ there exists $h\in G$ with $n^{j_l}_h\neq 0$ and $\overline{h}=\alpha \}$ and let $h\in G$ be such that $n_h^{j_t}\neq 0$ and $\overline{h}=\alpha$. The equation $E_{i_t,g}$ is 

$$\sum\limits_{1\le j\le m}\sum\limits_{k\in \occ(x_{i_t},r_j)} \epsilon_{k,r_j}n^j_{gs(k,r_j)}=0.$$

The maximum of $M_{i_t,j_t}$ is $\lambda=\overline{s(k,r_{j_t})}$ for certain $k\in \occ(x_{i_t},r_{j_t})$. Let $g=hs(k,r_{j_t})^{-1}$. Then, the term $n_h^{j_t}$ appears in the equation $E_{i_t, g}$. Since $n_h^{j_t}\neq 0$, there must be a second nonzero term in the equation. However, $\alpha=\overline{h}=\overline{g}+\lambda$ is the maximum among the weights of all the subindices in $E_{i_t,g}$. Therefore, any nonzero term must have subindex of weight $\alpha$. This cannot happen for terms with superindex $j_l$ and $l<t$ by definition of $t$ and this cannot happen for the other terms either, by the hypothesis on the multiplicity of $\lambda$ in $\bigcup\limits_{l=t}^m M_{i_t,j_l}$. 
\end{proof}

In the proof of Theorem \ref{main} we have only used part of the information given by the equations $E_{i,g}$. Namely, that if the sum of the terms in each equation is zero, then they are all zero or there are at least two nonzero terms. Therefore we have proved that the I-test guarantees that $\p$ satisfies the following property (p): If $\{n^j_g\}_{j,g}$ is a family of integer numbers with finite support, and in each equation $E_{i,g}$ all the terms are zero or there are at least two nonzero terms (we do not actually need the family to satisfy the equations), then all the $n^j_g$ must be trivial.

We will see that property (p) is equivalent to the notion of diagrammatic reducibility. Recall that a combinatorial 2-complex $K$ (for instance, the complex associated with a presentation) is said to be diagrammatically reducible (DR) if for each cell structure $C$ on $S^2$ and each combinatorial map $f:C\to K$, there exist distinct 2-cells $e^2,\widetilde{e}^2$ of $C$ with a common 1-face $e^1$ and a homeomorphism $h:e^2\to \widetilde{e}^2$ fixing $e^1$ pointwise such that $f|_{\widetilde{e}^2}h=f|_{e^2}$. We will say that the presentation $\p$ is DR if its standard complex is DR. This notion is stronger than asphericity. For basic definitions and applications of diagrammatic reducibility see \cite{Ger, Ger2}. In any case, for our purposes only the following characterization will be needed. This result, due to Corson and Trace, appeared in \cite[Theorem 2.4]{CT}. 

\begin{teo}[Corson-Trace] \label{corson}
A combinatorial $2$-complex $K$ is DR if and only if every finite subcomplex of the universal cover $\widetilde{K}$ collapses to a $1$-dimensional complex.
\end{teo}

To prove that $\p$ satisfies property (p) if and only if $K_\p$ is DR, we argue as follows. Suppose $K_\p$ is not DR. Then there exists a finite subcomplex $L$ of $\widetilde{K}_\p$ of dimension $2$ without free faces. We define a family $\{n^j_g\}_{j,g}$ associated to $L$. For each $1\le j\le m$ and $g\in G$ define $n^j_g=1$ if $e^2_{j,g}\in L$ and $n^j_g=0$ otherwise. Then $\{n^j_g\}_{j,g}$ is a nontrivial family with finite support which is not a solution to the system $\{E_{i,g}\}_{i,g}$, in general. However, since each $1$-cell $e^1_{i,g}$ is a proper face of $0$ or $\ge 2$ cells, by Remark \ref{terminos}, each equation $E_{i,g}$ has zero or at least two nonzero terms. Conversely, if $\{n^j_g\}_{j,g}$ is a nontrivial family with finite support and each equation $E_{i,g}$ has zero or at least two nonzero terms, then the subcomplex of $\widetilde{K}_\p$ generated by the cells $e^2_{j,g}$ such that $n^j_g\neq 0$, is finite, $2$-dimensional, and has no free faces. Then $K_\p$ is not DR. 
Therefore, we have proved

\begin{teo} \label{maindr}
If $\p$ satisfies the I-test, $K_\p$ is DR.
\end{teo}

Note that the I-test can only be applied to presentations $\p$ such that the presented group $G$ is indicable and the deficiency of $\p$ is non-negative, i.e. $n\ge m$. The positivity of the deficiency automatically implies indicability.

\begin{obs}
 Given a presentation $\p$, the set $V\subseteq \R^n$ of vectors $v$ orthogonal to every $q(r_j)$ is naturally described as the solution of a system of $m$ linear equations in $n$ variables. Given $v\in V$, it is easy to decide algorithmically whether $\p$ satisfies the I-test with respect to $v$. Moreover, if $\{v_1,v_2,\ldots, v_k\}$ is a basis of $V$, then $\p$ satisfies the I-test with respect to a vector $v=\lambda_1v_1+\lambda_2v_2+\ldots +\lambda_kv_k$ and orderings $j_1, j_2,\ldots, j_m$ and $i_1,i_2, \ldots, i_m$ of columns and rows if and only if certain linear inequalities on the $\lambda_i$ are satisfied. A system of equations of the form $\sum \lambda_i c_i \ge 0$ or $\sum \lambda_i c_i > 0$ has a solution if and only if the system obtained by replacing each $\sum \lambda_i c_i > 0$ by $\sum \lambda_i c_i \ge 1$ has a solution (if $\lambda=(\lambda_1, \lambda_2, \ldots, \lambda _k)\in \R^k$ solves the first one, for $N$ large enough $N\lambda$ solves the second one). The feasibility can be then decided with linear programming. Therefore it is algorithmically decidable whether a presentation satisfies the I-test for some vector $v$.      
\end{obs}

If $\p$ satisfies the I-test for $v$ and orderings $j_1, j_2,\ldots, j_m$ and $i_1,i_2, \ldots, i_m$ of columns and rows, then any partial orderings $j'_1, j'_2,\ldots, j'_l$ and $i'_1,i'_2,\ldots, i'_l$ ($l\le m$) satisfying Definition \ref{defiasph} for $1\le k\le l$ can be extended to orderings $j'_1, j'_2,\ldots, j'_m$ and $i'_1,i'_2,\ldots, i'_m$ such that $\p$ satisfies the I-test for those orderings.

\begin{obs}
If $\p$ satisfies the I-test, then any subpresentation of $\p$ does. In particular the Whitehead conjecture is true for presentations satisfying the I-test. The later follows also from Theorem \ref{maindr}.
\end{obs}

The following example shows a class of presentations which satisfy our test.
 
\begin{ej}\label{clasequecumple}
Let $n\in \mathbb{N}$ and let $$\p =\langle x_1,x_2, \ldots, x_n,x | x_ix\omega_i=x\tau_i, 1\le i\le n \rangle$$ where $\omega_i, \tau_i$ are words in which $x$ occurs only with positive exponent. Suppose further that for each $i$ the total exponent $exp(x,\omega_i)$ of $x$ in $\omega_i$ coincides with $exp(x,\tau_i)$. We show that $\p$ satisfies our test and therefore it is DR.

The vector $v=(0,0,\ldots ,0,1)$ is orthogonal to each $q(r_i)$, where $r_i=x_ix\omega_i(\tau_i)^{-1}x^{-1}$. The weight of $s(1,r_i)=r_i$ is $0$. For any other occurrence $k\in \occ(x_i,r_i)$ of $x_i$ in $r_i$, $\overline{s(k,r_i)}=exp(x,s(k,r_i))<0$. At the same time, all the elements in $M_{i,j}$ are negative for $j\neq i$. Then $\p$ satisfies the I-test with respect to $v$ and the natural order $1,2,\ldots, n$ of the columns and rows.
\end{ej}


\section{Relationship with previous weight tests}

Gersten's weight test provides a useful tool for proving diagrammatic reducibility (and hence asphericity) of group presentations \cite{Ger}. In \cite{HR} Huck and Rosebrock introduced a more general weight test which guarantees asphericity. We recall these criteria and exhibit examples of presentations which satisfy the I-test but not the previous weight tests.

As explained in \cite{HR}, given a presentation $\p=\langle x_1, x_2,\ldots ,x_n | r_1, r_2, \ldots, r_m\rangle$,  the \textit{Whitehead graph} $W_{\p}$ is the boundary of a regular neighborhood of the unique vertex of $K_{\p}$. In other words, $W_{\p}$ is the undirected graph (possibly with parallel edges) which can be described as follows. For each word $w=x_{i_1}^{\epsilon_1}x_{i_2}^{\epsilon_2} \ldots x_{i_l}^{\epsilon_l}$, with $\epsilon_i=\pm 1$, consider the graph $W_w$ whose vertex set is $\{x_1,-x_1,x_2,-x_2,\ldots, x_n,-x_n\}$ and with an edge $(\epsilon_kx_{i_k},-\epsilon_{k+1}x_{i_{k+1}})$ for each $1\le k\le l$ (subindices considered modulo $l$). Then $W_{\p}$ is the union of all the graphs $W_{r_j}$, $1\le j\le m$.

Suppose that no relator $r_j$ is a proper power. A \textit{weight function} on $W_{\p}$ is a real valued function $g$ on the edges of $W_{\p}$ (i.e. the \emph{corners} of $K_{\p}$). A weight function $g$ satisfies Huck and Rosebrock's weight test if

\begin{enumerate}

\item For each $1\le j\le m$, $\sum\limits_{e\in E(W_{r_j})} g(e)\le \len(r_j)-2$ and

\item For each simple cycle $z$ in $W_{\p}$, $\sum\limits_{e\in z} g(e)\ge 2$.

\end{enumerate}

Theorem 2.2 in \cite{HR} claims that the existence of a weight function satisfying the test implies vertex asphericity, a concept which in turn implies asphericity of $\p$. Gersten's original weight test requires a stronger hypothesis on the weight function $g$ and implies diagrammatic reducibility (see \cite[Theorem 4.7]{Ger}).

The following example shows a presentation which satisfies the I-test while no weight function $g$ satisfies the weight test.

\begin{ej} \label{ejweighttest}
Let $\p=\langle x,y| x^3yxy\rangle$. Then $\p$ satisfies the I-test but it does not satisfy the weight test. The Whitehead graph of $\p$ appears in Figure \ref{wei}.

\begin{figure}[h] 
\begin{center}
\includegraphics[scale=0.5]{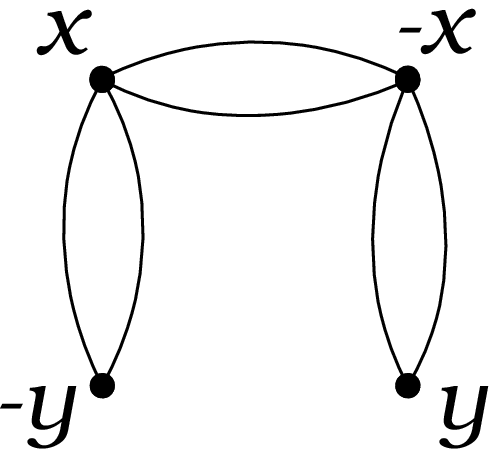}

\caption{The Whitehead graph of $\p=\langle x,y| x^3yxy\rangle$.}\label{wei}
\end{center}
\end{figure}
Suppose a weight function $g$ satisfies the weight test. Condition (2) implies that the sum of the weights of the two edges from $x$ to $-y$ is at least $2$. The same holds for the two edges from $x$ to $-x$ and from $-x$ to $y$. Then condition (1) gives $6\le \sum\limits_{e\in W_{\p}} g(e)\le \len(x^3yxy)-2=4$, a contradiction.

The vector $v=(1,-2)$ is orthogonal to $q(x^3yxy)=(4,2)$. The coefficient of the second row in the weight matrix is $M_{2,1}=\{-3,-2\}$, so $\p$ satisfies the I-test.

Similarly, the presentation $\p=\langle x,y | x^3y^3xy \rangle$ provides another example of a presentation which satisfies the I-test but not the weight test. 
\end{ej}

The presentations of Example \ref{clasequecumple}, which are proved to be DR by the I-test, do not satisfy in general the weight tests. Take for instance $\omega_1=x_1^{-1}x$ and $\tau_1=x_1xx_1$. Then the weight test fails independently of the choices of $\omega_i$, $\tau_i$ for $2\le i\le n$.

\section{Adian presentations and LOGs}

We now apply our test to generalize some results of Howie and Gersten on LOTs and Adian presentations. This will produce extensive families of DR presentations.

Recall that a labeled oriented graph (LOG) $\Gamma$ consists of two finite sets $V$, $E$, and three maps $\iota, \tau, \lambda: E\to V$. The elements of $V$ are called vertices and the elements of $E$, edges. If $e\in E$, $\iota (e)$ is called the initial vertex of $e$, $\tau (e)$ its terminal vertex, and $\lambda (e)$ its label. A LOG $\Gamma$ is called a labeled oriented tree (LOT) if the underlying graph is a tree. Associated with a LOG $\Gamma$ there is a presentation. An edge $e$ of $\Gamma$ with initial vertex $x$, terminal vertex $y$ and label $z$, has an associated relator $r_e=xzy^{-1}z^{-1}$. The presentation corresponding to a LOG $\Gamma$ with vertex set $V=\{x_1,x_2,\ldots , x_n\}$ and edge set $E=\{e_1, e_2, \ldots, e_m\}$ is $\p _{\Gamma}=\langle x_1, x_2,\ldots , x_n \ | \ r_{e_1}, r_{e_2}, \ldots , r_{e_m}\rangle$.

%

Given a LOG $\Gamma$, consider the undirected graph $I(\Gamma)$ with the same vertex set as $\Gamma$ and with an edge $(x,y)$ for each edge of $\Gamma$ with label $x$ and terminal vertex $y$. Similarly, define $T(\Gamma)$ with the same vertex set and an edge $(x,y)$ for each edge of $\Gamma$ with initial vertex $x$ and label $y$. The following result by Howie (\cite[Corollary 10.2]{Ho4}) says that these graphs can be used to prove asphericity of $\Gamma$.

\begin{teo}[Howie] \label{howie}
If either $I(\Gamma)$ or $T(\Gamma)$ is a tree, $\p _{\Gamma}$ is aspherical.
\end{teo}

We will prove a generalization of this result.

\begin{defi} \label{defo}
Let $\Gamma$ be a LOG and let $\Gamma'$ be a sub-LOG (i.e. an \emph{admissible} subgraph, in Howie's terminology \cite{Ho4}). We say that there is a \textit{deforestation of type IL/T} from $\Gamma$ to $\Gamma'$ if the number of edges in $\Gamma \smallsetminus \Gamma'$ equals the number of vertices in $\Gamma \smallsetminus \Gamma'$ and there exists an ordering $e_1,e_2, \ldots , e_m$ of the edges of $\Gamma \smallsetminus \Gamma'$ and an ordering $x_1,x_2, \ldots , x_m$ of the vertices of $\Gamma \smallsetminus \Gamma'$ such that for each $1\le j\le m$ one of the following holds:  


\begin{enumerate}

\item[(IL)] The number of edges in $\Gamma_{j-1}=\Gamma \smallsetminus \{e_1, e_2,\ldots ,e_{j-1}\}$ which have initial vertex $x_j$ plus the number of edges in $\Gamma_{j-1}$ having label $x_j$ is one. Moreover, this unique edge with initial vertex or label $x_j$ is $e_j$.



\item[(T)] $x_j$ is the terminal vertex of a unique edge in $\Gamma_{j-1}$, this edge being $e_j$. Moreover, there is no edge in $\Gamma$ whose initial vertex or label is $x_j$.
\end{enumerate}

A \textit{deforestation of type TL/I} is defined similarly by interchanging the words ``initial'' and ``terminal''. A LOG $\Gamma$ is \textit{deforestable} if there is a deforestation of type IL/T to a discrete subgraph or a deforestation of type TL/I to a discrete subgraph. Of course, a LOG $\Gamma$ is deforestable if and only if its opposite $\Gamma^{op}$ is deforestable.

\end{defi}


\begin{prop} \label{defimplicadr}
If $\Gamma$ is deforestable LOG, $\p _\Gamma$ satisfies the I-test.
\end{prop}
\begin{proof}
Suppose there is a deforestation of type IL/T from $\Gamma$ to a discrete sub-LOG $\Gamma'$. Let $e_1,e_2, \ldots, e_{m}$ and $x_1,x_2, \ldots, ,x_{m}$ be as in Definition \ref{defo}. Let $x_{m+1},x_{m+2},\ldots, x_n$ be the vertices of $\Gamma'$. Then $\p _\Gamma=\langle x_1,x_2, \ldots, x_{m}, x_{m+1}, \ldots , x_n $ $|$ $ r_1, r_2, \ldots , r_m \rangle$, where $r_j$ is the relation associated to $e_j$.

Suppose $e_j$ is an edge of $\Gamma$ with initial vertex $x$, terminal vertex $y$ and label $z$ (we do not require them to be different). Then $r_j=xzy^{-1}z^{-1}$. In particular $v=(1,1,\ldots ,1)$ is orthogonal to each $q(r_j)$. We will prove that $\p _\Gamma$ satisfies the $I$-test for $v$ and the natural order of the columns and rows of $M$. The coefficient $M_{i,j}$ consists of a (possibly empty) family which contains the elements $-1$ and $0$ with certain multiplicity. If $x_i$ is the terminal vertex of $e_j$, then $-1$ is an element of this family, if $x_i$ is the initial vertex, then $0\in M_{i,j}$, and if $x_i$ is the label, $-1,0\in M_{i,j}$.

Let $1\le j\le m$. If condition (IL) holds, then $0$ is an element of $M_{j,j}$ of multiplicity one and $0\notin M_{j,k}$ for each $k>j$. If condition (T) holds, then $M_{j,j}=\{-1\}$, $M_{j,k}=\emptyset$ for each $k>j$ and $M_{j,k}=\emptyset$ or $\{-1\}$ for each $1\le k\le m$. This proves that $\p _\Gamma$ satisfies our test.

Finally, if there is a deforestation of type TL/I from $\Gamma$ to a discrete sub-LOG, $\p _{\Gamma}$ satisfies the I-test for $v=(-1,-1, \ldots , -1)$. 
\end{proof}

\begin{prop}
Let $\Gamma$ be a LOG. If $T(\Gamma)$ or $I(\Gamma)$ has no cycles, $\Gamma$ is deforestable.
\end{prop}
\begin{proof}
Suppose $T(\Gamma)$ has no cycles and that $x_1$ is a leaf of $T(\Gamma)$. Then there is a unique edge $e_1$ of $\Gamma$ whose initial vertex or label is $x_1$, but not both simultaneously.

Let $\Gamma_1=\Gamma \smallsetminus \{e_1\}$ and let $T_1$ be the subgraph of $T(\Gamma)$ induced by the vertices different from $x_1$. By induction suppose defined $x_1,x_2, \ldots , x_j$ and $e_1, e_2, \ldots , e_j$. Let $\Gamma_j=\Gamma \smallsetminus\{e_1, e_2,\ldots , e_j\}$ and let $T_{j}$ be the subgraph of $T(\Gamma)$ induced by the vertices different from $x_1,x_2,\ldots ,x_j$. Let $x_{j+1}$ be a leaf of $T_{j}$. Then there is a unique edge $e_{j+1}$ of $\Gamma_j$ whose initial vertex or label is $x_{j+1}$. This process ends when $T_j$ is discrete for some $j$, say $j=m$, and, then, $\Gamma_m$ is also discrete. The orders $e_1,e_2,\ldots, e_m$ and $x_1,x_2,\ldots, x_m$ constructed show that there is a deforestation of type IL/T from $\Gamma$ to $\Gamma'=\Gamma_m\smallsetminus \{x_1,x_2, \ldots , x_m\}$. Therefore, $\Gamma$ is deforestable. Note that condition (T) is never used. If $I(\Gamma)=T(\Gamma^{op})$ has no cycles, $\Gamma^{op}$ is deforestable, and then so is $\Gamma$.
\end{proof}

\begin{ej}
The following LOT is deforestable but both $I(\Gamma)$ and $T(\Gamma)$ have cycles.
\begin{figure}[h] 
\begin{center}
\includegraphics[scale=0.5]{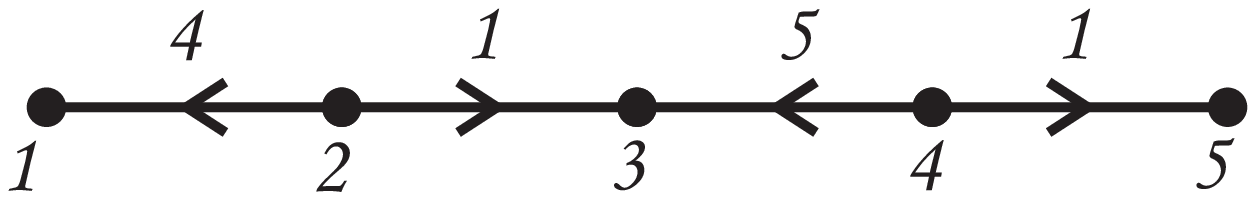}
\end{center}
\end{figure}

There is a deforestation of type IL/T from $\Gamma$ to a one-vertex LOT. An order of the vertices which satisfies Definition \ref{defo} is $5, 3, 2, 1$.
\end{ej}

In Section \ref{seccionextended} we will extend Proposition \ref{defimplicadr} to a wider class of LOGs.

\bigskip

An \textit{Adian} presentation is a presentation $\p=\langle x_1, x_2, \ldots, x_n | U_j=V_j, j \in J \rangle$ where $U_j$ and $V_j$ are non-trivial positive words. Gersten defines the \textit{left graph} $L(\p)$ of $\p$ as the undirected graph whose vertices are the first letters of the words $U_j$ and the first letters of the words $V_j$. For each $j\in J$ there is an edge (the $j$-th edge) from the first letter of $U_j$ to the first letter of $V_j$. The \textit{right graph} $R(\p)$ is defined similarly considering the last letters of the words $U_j, V_j$. Gersten proved  the following result \cite[Proposition 4.15]{Ger}.

\begin{prop}[Gersten] \label{adian}
Let $\p=\langle x_1, x_2, \ldots, x_n | U_j=V_j, j \in J \rangle$ be an Adian presentation such that $\len(U_j)=\len(V_j)$ for each $j\in J$. If either $L(\p)$ or $R(\p)$ has no cycles, $K_\p$ is DR. 
\end{prop}

Note that the presentation $\p _{\Gamma}$ associated to a LOG $\Gamma$ is Adian, and in this case $L(\p)=T(\Gamma)$ and $R(\p)=I(\Gamma)$. Hence, Proposition \ref{adian} generalizes Theorem \ref{howie}. We will now use our methods to prove a generalization of Proposition \ref{adian}.

Let $\p=\langle x_1, x_2, \ldots, x_n | U_j=V_j, j \in J \rangle$ be an Adian presentation. We label each edge of $L(\p)$ with a family of generators of $\p$ in the following way. Given a generator $x=x_i$, consider those words $U_j$, $V_j$ containing $x$, in which $x$ occurs in the left-most position among all the words $U_k,V_k$. Then label the corresponding $j$-th edges of $L(\p)$ with $x$. An edge may have various labels. In particular, if $x$ occurs simultaneously in $U_j$ and $V_j$ in the left-most position among all the $U_k$ and $V_k$, then the $j$-th edge will have the label $x$ twice. Note that every generator which appears in some relator will be label of at least one edge of $L(\p)$. Note also that the vertices adjacent to an edge are always labels of that edge. These labels are called \emph{trivial} and they are not written in the graphical representation of $L(\p)$. Similarly, we label the edges of the graph $R(\p)$ with the generators that occur at the right-most positions among all the words.

\begin{ej} \label{ejdiscretizable}
Let $\p=\langle x,y,z,t,u,v | xuz^2=yutv, yu=zv, y^3=tx^2, tv^2=zxu, tx=zy \rangle$. Note that both $L(\p)$ and $R(\p)$ have cycles. Figure \ref{izqder} shows the labeled left and right graphs of $\p$.
\begin{figure}[h] 
\begin{center}
\includegraphics[scale=0.5]{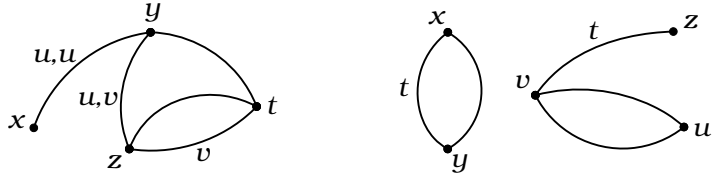}

\caption{The labeled left and right graphs of $\p$.}\label{izqder}
\end{center}
\end{figure}

\end{ej}  

\begin{defi}
Let $G$ be an undirected graph in which each edge is labeled with a family of elements in a set $S$. If $e$ is an edge of $G$ which is labeled with a family that contains an element $s\in S$, which appears only once in total in all the labels of $G$, then we can remove the edge $e$ from $G$. This removal is called a \textit{deletion}. The graph $G$ is called \textit{discretizable} if a graph with no edges can be obtained from $G$ by a sequence of deletions.   
\end{defi}

\begin{obs}
If $\p$ is an Adian presentation and $L(\p)$ has no cycles, then $L(\p)$ is discretizable. Note that if $x=x_i$ is a leaf of $L(\p)$, then the edge incident to $x$ is the unique edge labeled with $x$. Thus we can delete leaf edges one at the time. Analogously, if $R(\p)$ has no cycles, it is discretizable. Example \ref{ejdiscretizable} shows a presentation $\p$ such that $L(\p)$ is discretizable, although it has cycles. Since $x$ is a leaf of $L(\p)$, the edge $xy$ can be deleted. Now, the generator $u$ appears in a unique label, so edge $yz$ can be removed. Then we can delete $yt$, then the edge containing $v$ in its label, and finally the last edge. By Theorem \ref{teodiscretizables} below, this presentation is DR. Note that the right graph is not discretizable: although the edge $zv$ and both edges between $x$ and $y$ can be deleted, the edges between $u$ and $v$ cannot be removed.  
\end{obs}

\begin{teo} \label{teodiscretizables}
Let $\p=\langle x_1, x_2, \ldots, x_n | U_j=V_j, j \in J \rangle$ be an Adian presentation such that $\len(U_j)=\len(V_j)$ for each $j\in J$. If either $L(\p)$ or $R(\p)$ is discretizable, $\p$ satisfies the I-test.
\end{teo}
\begin{proof}
Suppose $L(\p)$ is discretizable. We prove that $\p$ satisfies the $I$-test for the vector $v=(1,1,\ldots , 1)$. If $x_i$ is the $k$-th letter of $U_j$, the weight $\overline{s(k, U_jV_j^{-1})}=1-k$ is an element of $M_{i,j}$. If $x_i$ is the $k$-th letter of $V_j$, then $\overline{s(2\len(U_j)-k+1, U_jV_j^{-1})}=1-k \in M_{i,j}$. Thus, if $x_i$ appears only once in the labels of $L(\p)$, in the $j$-th edge, then $x_i$ is the $k$-th letter of $U_j$ or of $V_j$ for some $k$ (but not of both), the first $k-1$ letters of $U_j$ and $V_j$ are different from $x_i$, and the first $k$ letters of $U_{j'}$ and $V_{j'}$ are different from $x_i$ for every $j'\neq j$. Then $1-k\in M_{i,j}$ is the maximum of the whole $i$-th row in the weight matrix $M$ and it appears with multiplicity one. The order of deletion of the edges of $L(\p)$ induces an order in the columns and $|J|$ of the rows which proves that $\p$ satisfies the I-test.
If $R(\p)$ is discretizable, $\p$ satisfies our test for $v=(-1,-1,\ldots ,-1)$. 
\end{proof}

\begin{obs}
If $\Gamma$ is a LOG, then $L(\p_\Gamma)$ or $R(\p_\Gamma)$ is discretizable if and only if $\Gamma$ is deforestable. 
\end{obs} 

\begin{obs}
	These ideas can be easily extended to study diagrammatic reducibility of more general presentations $\p=\langle x_1,x_2,\ldots, x_n | r_1,r_2,\ldots ,r_m \rangle$ such that the total exponent of each relator $r_j$ is $0$ and each proper final segment $(r_j)^{(k)}$ has negative total exponent. In this case, define $L(\p)$ to be the undirected graph whose vertices are the first and last letters of the relators $r_j$ and with an edge (the $j$-th edge) from the first to the last letter of $r_j$ for each $1\le j\le m$. The labels of the edges in $L(\p)$ are the following. For each generator $x_i$ consider all the total exponents of the words $s(k,r_j)$ with $1\le j\le m$ and $k\in \occ(x_i,r_j)$ and label the $j$-th edge of $L(\p)$ with $x_i$ if the maximum of those exponents is attained at $r_j$, once for each time this number is attained at $r_j$. Then, if $L(\p)$ is discretizable, $\p$ is aspherical.
\end{obs}

\section{Ivanov's asphericity conjecture}

In \cite{Iva} Ivanov proposed the following conjecture.

\begin{conj}[Ivanov] \label{conjivanov}
Let $\p=\langle x_1,x_2,\ldots, x_n | r_1,r_2, \ldots, r_m \rangle$ be an aspherical presentation and let $\mathcal{Q}=\langle x_1,x_2,\ldots, x_n,x | r_1,r_2, \ldots, r_m, r \rangle$ be such that:
\begin{itemize}
\item The total exponent of $x$ in $r$ is non-zero,
\item The group $H$ presented by $\p$ naturally embeds in the group $G$ presented by $\mathcal{Q}$, and
\item $G$ is torsion-free.
\end{itemize}
Then $\mathcal{Q}$ is aspherical.
\end{conj}

It is unknown whether the first hypothesis implies the second one (this is in fact the statement of the Kervaire Conjecture).
Ivanov's conjecture is related to the Kaplansky conjecture on zero divisors of a group ring. If Conjecture \ref{conjivanov} is false, there is a torsion free group $G$ such that $\mathbb{Z}G$ has zero divisors (\cite[Theorem 1]{Iva}). In this section we will prove two results related to Conjecture \ref{conjivanov}.

\begin{lema} \label{lemaivanov}
Let $\p=\langle x_1,x_2,\ldots, x_n | r_1,r_2, \ldots, r_m \rangle$ be an aspherical presentation and let $\mathcal{Q}=\langle x_1,x_2,\ldots, x_n,x | r_1,r_2, \ldots, r_m, r \rangle$ be such that the group $H$ presented by $\p$ naturally embeds in the group $G$ presented by $\mathcal{Q}$. Moreover, suppose that for any solution $\{n_g^j\}_{1\le j\le m, g\in G}\cup \{n_g\}_{g\in G}$ of the system $\{E_{i,g}\}_{1\le i\le n, g\in G}\cup \{E_g\}_{g\in G}$ with finite support, we have that $n_g=0$ for every $g\in G$. Here $n_g$ denotes the coefficients corresponding to the relator $r$ and $E_g$ the equations associated to the generator $x$. Then $\mathcal{Q}$ is aspherical.
\end{lema}
\begin{proof}
Let $\{n_g^j\}_{1\le j\le m, g\in G}\cup \{n_g\}_{g\in G}$ be a solution of the system $\{E_{i,g}\}_{1\le i\le n, g\in G}\cup \{E_g\}_{g\in G}$ with finite support. We want to show that all the $n_g^j$ and $n_g$ are zero. By hypothesis $n_g=0$ for every $g\in G$. Let $k\in G$. Let $\iota :H\to G$ be the natural embedding. For each $1\le j\le m$ and each $h\in H$, define $\widetilde{n}_h^j=n_{k\iota(h)}^j$. Since $n_g=0$ for every $g\in G$, the equation $\widetilde{E}_{i,h}$ associated to the presentation $\p$, reduces to equation $E_{i,k\iota(h)}$ associated to $\mathcal{Q}$. Therefore, $\{\widetilde{n}_h^j\}_{1\le j\le m, h\in H}$ is a solution of the system $\{\widetilde{E}_{i,h}\}_{1\le i\le n, h\in H}$. Moreover, since $n_g^j=0$ except for finitely many $(g,j)$ and the kernel of $\iota$ is finite (trivial, in fact), only finitely many $\widetilde{n}_h^j$ are non-zero. Theorem \ref{ecuaciones} implies then that $\widetilde{n}_h^j=0$ for every $h\in H$ and $1\le j\le m$. In particular, $n_k^j=\widetilde{n}_1^j=0$ for every $j$. Since $k\in G$ is arbitrary, by Theorem \ref{ecuaciones}, $\mathcal{Q}$ is aspherical.
\end{proof}   

Recall that a group $G$ is said to have the \textit{unique product property (upp)} if for any two non-empty finite subsets $A,B\subseteq G$, there exists $g\in G$ such that $gA\cap B$ has exactly one element. It is easy to see that if $G$ has the upp, then it is torsion-free and the group ring $\mathbb{Z}G$ has no zero divisors (see \cite{st} for examples of torsion-free groups without the upp). The result of Ivanov \cite[Theorem 1]{Iva} mentioned above implies then that his conjecture is true if $G$ has the upp. We give next a direct proof of this result using Lemma \ref{lemaivanov}.

\begin{prop} \label{upp}
Ivanov's conjecture is true if the group $G$ presented by $\mathcal{Q}$ has the unique product property.
\end{prop}
\begin{proof}
Let $\{n_g^j\}_{1\le j\le m, g\in G}\cup \{n_g\}_{g\in G}$ be a solution with finite support of the system $\{E_{i,g}\}_{1\le i\le n, g\in G}\cup \{E_g\}_{g\in G}$ associated to $\mathcal{Q}$. Since the generator $x$ does not appear in the relators $r_j$, the equation $E_{g}$ is $$\sum\limits_{k\in \occ(x,r)} \epsilon_{k,r}n_{gs(k,r)}=0,$$ where $\occ(x,r)$ denotes the set of occurrences of $x$ in $r$. Since the total exponent $\sum{\epsilon_{k,r}}$ of $x$ in $r$ is non-zero, the equation above is not trivial, and it can be rewritten as 
\begin{equation}
\sum\limits_{i=1}^p a_in_{gs_i}=0,
\label{eq2}
\end{equation}
 where $p\ge 1$, $0\neq a_i\in \mathbb{Z}$ for every $1\le i\le p$, and $s_1, s_2, \ldots, s_p$ are different elements of $G$. Let $A=\{s_1,s_2,\ldots, s_p\}$ and $B=\{g\in G| n_g\neq 0\}$. Suppose $B$ is non-empty. By hypothesis, there exists $g\in G$ such that $gA\cap B$ has a unique element, say $h\in G$. Then equation (\ref{eq2}) above says that $n_h=0$, a contradiction. Thus, $B=\emptyset$ and Lemma \ref{lemaivanov} implies that $\mathcal{Q}$ is aspherical.
\end{proof}

As explained in \cite{Iva}, Howie's results show that Conjecture \ref{conjivanov} is true when $G$ (equivalently $H$) is locally indicable, that is, if every nontrivial finitely generated subgroup of $G$ is indicable (this follows also from Proposition \ref{upp}). The next result improves this. Concretely, if $x$ occurs $k$ times in $r$, then we only have to check that certain $2^k-\frac{k(k+1)}{2}-1$ subgroups of $G$ are indicable. In the case of $4$ occurrences, for instance, the asphericity of $\mathcal{Q}$ would follow from the indicability of a subgroup of $G$ generated by $3$ elements and four subgroups generated by $2$ elements each. Recall that a group $G$ is said to be left-orderable if there exists a total order $\le$ in the underlying set of $G$ such that $h\le h'$ implies $gh\le gh'$ for every $h,h',g\in G$.

\begin{teo} \label{ivanovli}
Under the hypothesis of Conjecture \ref{conjivanov}, suppose that for every subset $S\subseteq \{s(k,r)\in G | k\in \occ(x,r)\}$ with more than two elements, the subgroup $G_S$ of $G$ generated by the elements $s^{-1}t$ for $s,t\in S$, is indicable, or, more generally, there exists a nontrivial homomorphism $G_S\to L_S$ to a left-orderable group $L_S$. Then $\mathcal{Q}$ is aspherical. 
\end{teo}
\begin{proof}
For simplicity we assume that each subgroup $G_S$ in the statement is indicable. The proof of the general case is identical.
Suppose $\{n_g^j\}_{1\le j\le m, g\in G}\cup \{n_g\}_{g\in G}$ is a solution with finite support of the system $\{E_{i,g}\}_{1\le i\le n, g\in G}\cup \{E_g\}_{g\in G}$. Assume there exists $g_0\in G$ such that $n_{g_0}\neq 0$.
Similarly as in the proof of Proposition \ref{upp}, let 
\begin{equation*}
\sum\limits_{i=1}^p a_in_{gs_i}=0, \tag{\ref{eq2}}
\end{equation*}
be the equation $E_g$ associated to $x$ and $g\in G$, where $p\ge 1$, $a_i$ are nonzero integers and $s_1,s_2,\ldots, s_p$ are pairwise different elements of $G$. Then $\{s_1,s_2,\ldots, s_p\}\subseteq \{s(k,r)\in G | k\in \occ(x,r)\}$. Let $G_1=\langle s_i^{-1}s_j | 1\le i,j\le p \rangle$. If $p_1=p>1$, by hypothesis there exists a nontrivial homomorphism $\varphi _1:G_1 \to \Z$ (for $p=2$ this follows from the fact that $G$ is torsion-free). Let $m_1=\min \{\varphi_1(g) | g\in G_1 \textrm{ and } n_{g_0g}\neq 0\}$. Note that $m_1$ is well-defined since $n_{g_0}\neq 0$ and only finitely many $n_g$ are non-zero. Let $g_1\in G_1$ be such that $n_{g_0g_1}\neq 0$ and $\varphi_1(g_1)=m_1$. Let $M_1=\max \{\varphi_1(s_1), \varphi_1(s_2), \ldots, \varphi_1(s_{p_1})\}$. By reordering the $s_i$ if needed, there exists $1\le p_2\le p_1$ such that $\varphi_1(s_i)=M_1$ for $1\le i\le p_2$ and $\varphi_1(s_i)<M_1$ for $p_2< i\le p_1$. Since $\varphi_1$ is non-trivial, $p_2<p_1$. Let $G_2=\langle s_i^{-1}s_j | 1\le i,j\le p_2\rangle$. If $p_2>1$, by hypothesis there exists a nontrivial homomorphism $\varphi_2:G_2\to \Z$. Let $m_2=\min \{\varphi_2(g) | g\in G_2 \textrm{ and } n_{g_0g_1g}\neq 0\}$. Let $g_2\in G_2$ be such that $n_{g_0g_1g_2}\neq 0$ and $\varphi_2(g_2)=m_2$. Let $M_2=\max \{\varphi_2(s_1), \varphi_2(s_2), \ldots, \varphi_2(s_{p_2})\}$. By reordering the $s_i$, there exists $1\le p_3\le p_2$ such that $\varphi_2(s_i)=M_2$ for $1\le i\le p_3$ and $\varphi_2(s_i)<M_2$ for $p_3< i\le p_2$. Since $\varphi_2$ is non-trivial, $p_3<p_2$. This process ends when $p_{l+1}=1$ for some $l$. Then $m_l=\min \{ \varphi_l(g) | g\in G_l \textrm{ and } n_{g_0g_1 \ldots g_{l-1}g}\neq 0\}$, the minimum attained by some $g_l\in G_l$ and the maximum $M_l=\max \{\varphi_l(s_1), \varphi_l(s_2), \ldots, \varphi_l(s_{p_l})\}$ is only attained by $s_1$. Note that $G_1\supseteq G_2 \supseteq \ldots \supseteq G_l$.

Let $g=g_0g_1\ldots g_ls_1^{-1}\in G$. Equation (\ref{eq2}) becomes then $$\sum\limits_{i=1}^p a_in_{g_0g_1\ldots g_ls_1^{-1}s_i}=0.$$

Let $i>p_2$. Since $\varphi_1$ is trivial in $G_2$, $\varphi_1(g_1g_2\ldots g_ls_1^{-1}s_i)=\varphi_1(g_1)-\varphi_1(s_1)+\varphi_1(s_i)<\varphi_1(g_1)=m_1$, thus $n_{g_0g_1\ldots g_ls_1^{-1}s_i}=0$. Then

$$\sum\limits_{i=1}^{p_2} a_in_{g_0g_1\ldots g_ls_1^{-1}s_i}=0.$$

If $i>p_3$, then $\varphi_2(g_2g_3\ldots g_ls_1^{-1}s_i)=\varphi_2(g_2)-\varphi_2(s_1)+\varphi_2(s_i)<\varphi_2(g_2)=m_2$, thus $n_{g_0g_1\ldots g_ls_1^{-1}s_i}=0$ and

$$\sum\limits_{i=1}^{p_3} a_in_{g_0g_1\ldots g_ls_1^{-1}s_i}=0.$$

Following this reasoning we conclude $a_1n_{g_0g_1\ldots g_l}=0$, which contradicts the definition of $g_l$. The contradiction arose from the faulty assumption that there exists $g_0\in G$ with $n_{g_0}\neq 0$. By Lemma \ref{lemaivanov}, $\mathcal{Q}$ is aspherical.

\end{proof}

The following result is a weak version of the conjecture, where we are allowed to perturb the relation $r$.

\begin{teo}\label{ivanovp}
Let $\p=\langle x_1,x_2,\ldots, x_n | r_1,r_2, \ldots, r_m \rangle$ be an aspherical presentation of an indicable group $H$ and let $\mathcal{Q}=\langle x_1,x_2,\ldots, x_n,x | r_1,r_2, \ldots, r_m, r \rangle$ be such that the total exponent of $x$ in $r$ is non-zero. Then there exists a cyclic permutation $r'$ of $r$, $1\le i\le n$ and $M_0\in \mathbb{N}$ such that, for every integer $M$ with $|M|\ge M_0$, the perturbed presentation $$\mathcal{Q}'=\langle x_1,x_2,\ldots, x_n,x | r_1,r_2, \ldots, r_m, x_i^Mr' \rangle $$ is aspherical if $H$ naturally embeds in the group $G'$ presented by $\mathcal{Q}'$.  
\end{teo}
\begin{proof}
Let $\beta$ be the total exponent of $x$ in $r$. Replacing $r$ by $r^{-1}$ if needed, we can assume $\beta >0$. Let $\occ(x,r)=\{k_1,k_2,\ldots, k_p\}$ be the set of occurrences of $x$ in $r$, with $1\le k_1(r) < k_2(r) < \ldots < k_p(r)\le \len (r)$. Consider the sequence $$S(r): a_1(r), a_2(r), \ldots, a_p(r)$$ where $a_i$ is the total exponent of $x$ in the subword $s(k_i,r)$. We will prove the following 

Claim: there exists a cyclic permutation $r'$ of $r$ such that $k_1(r')=1$, $\epsilon_{1,r'}=+1$ (that is, that the first letter of $r'$ is $x^{+1}$), and $a_i(r')<a_1(r')=\beta$ for every $i>1$.

If $r'$ is a cyclic permutation of $r$ which begins with $x^{+1}$, then it is easy to see that

(i) The sequence $S(r')$ begins with $a_1(r')=\beta$,

(ii) $|a_{i+1}(r')-a_{i}(r')|\le 1$ for every $1\le i<p$,

(iii) $a_p(r')$ is $0$ or $1$ (it is $0$ if $\epsilon_{k_p(r'),r'}=-1$ and $1$ if $\epsilon_{k_p(r'),r'}=1$).

Suppose that for a given cyclic permutation $r'=r'_1$ of $r$ which begins with $x^{+1}$, there exists $1<i\le p$ such that $a_i(r_1)\ge \beta$. Assume $i=i_1$ is maximum with that property. Then by (ii) and (iii), $a_i(r')=\beta$. We will show that $\epsilon_{k_i(r'),r'}=1$. Assume on the contrary that $\epsilon_{k_i(r'),r'}=-1$, then $i\neq p$ by (iii), since $\beta\neq 0$. Now there are two cases: $\epsilon_{k_{i+1}(r'),r'}=-1$ or $\epsilon_{k_{i+1}(r'),r'}=1$. In the first case, $a_{i+1}(r')=a_i(r')+1=\beta+1$, and in the second, $a_{i+1}(r')=a_i(r')=\beta$. Any of these contradicts the maximality of $i$. We conclude then that $\epsilon_{k_i(r'),r'}=1$. Note that the total exponent of $x$ in the subword $w=w_1$ given by the first $k_i(r')-1$ letters of $r'$ is $0$. We can consider then the cyclic permutation $r'_2$ of $r'=r'_1$ which begins with the $k_{i_1}(r'_1)$-th letter of $r'_1$ and repeat this argument. If $r'_2$ does not satisfy the property of the Claim, then there exists $1<i_2\le p$ such that the total exponent of $x$ in the subword $w_2$ given by the first $k_{i_2}(r'_2)-1$ letters of $r_2$ is $0$. If the Claim is false, there is a sequence $w_1,w_2, \ldots$ of subwords of cyclic permutations of $r$, all of which have zero total exponent for $x$. Moreover, the infinite word $w_1w_2w_3\ldots$ coincides with the infinite word $r'r'r'\ldots$. There exist then $k,l\in \mathbb{N}$ with $k<l$ such that $w_kw_{k+1}\ldots w_l$ is a nontrivial power $(r'_0)^N$ of a cyclic permutation $r'_0$ of $r$. But this is absurd since the total exponent of $x$ in $w_kw_{k+1}\ldots w_l$ is $0$, while the exponent of $x$ in $(r'_0)^N$ is $N\beta>0$. This finishes the proof of the Claim.

Let $r'$ be a cyclic permutation of $r$ which satisfies the properties stated in the Claim. Let $q:F(x_1,x_2,\ldots, x_n)\to \mathbb{Z}^n$ and $q':F(x_1,x_2,\ldots, x_n, x)\to \mathbb{Z}^{n+1}$ denote the abelianization maps. Since $H$ is indicable, there exists a nonzero vector $v=(v_1,v_2,\ldots, v_n)\in \R^n$ orthogonal to each $q(r_j)$. Let $1\le i\le n$ be such that $v_i\neq 0$ and let $M\in \mathbb{Z}$. Define $\widetilde{r}=x_i^Mr'$, let $\mathcal{Q}'$ be the presentation defined in the statement of the theorem and let $G'$ be the group presented by $\mathcal{Q}'$. Let $q'(r)=(\beta_1, \beta_2, \ldots, \beta_i, \ldots, \beta_n, \beta)$. Then $q'(\widetilde{r})=(\beta_1, \beta_2, \ldots, \beta_i+M, \ldots , \beta_n, \beta)$. Define $\alpha_M=-\frac{\sum v_l\beta_l +Mv_i}{\beta}$. Then $v_M'=(v_1,v_2,\ldots, v_n,\alpha_M)$ $=$ $(v_1, v_2,\ldots, v_n,\alpha_0)-(0,0,\ldots, 0, \frac{Mv_i}{\beta})$ is orthogonal to each $q'(r_j)$ and to $q'(\widetilde{r})$.

Let $g\in G'$ and let $E_{g}$ $$0=\sum\limits_{k\in \occ(x,\widetilde{r})} \epsilon_{k,\widetilde{r}}n_{gs(k,\widetilde{r})}=\sum\limits_{k\in \occ(x,r')} \epsilon_{k,r'}n_{gs(k,r')}=\sum\limits_{j=1}^p \epsilon_{k_j(r'),r'} n_{gs(k_j(r'),r')}$$ be the equation associated to $x$ and $g$ in $\mathcal{Q}'$.

The weight of $s(k_j(r'),r')$ is $\overline{s(k_j(r'),r')}=\langle q'(s(k_j(r'),r')), v_M'\rangle=\langle q'(s(k_j(r'),r')), v_0'\rangle-\langle q'(s(k_j(r'),r')), (0,0,\ldots, 0, \frac{Mv_i}{\beta})\rangle=c_j-\frac{Mv_i}{\beta}a_j(r')$, where $c_j$ is a constant which does not depend on $M$, and $a_j(r')$ is the total exponent of $x$ in $s(k_j(r'),r')$, which appears in the Claim above. By the Claim $a_j(r')<a_1(r')=\beta$ for every $j>1$. Then, if $|M|$ is big enough, the family of weights $$\{\overline{s(k_j(r'),r')} \ | \ 1\le j\le p\} $$ has a maximum or a minimum with multiplicity one. Then, the ideas of Section \ref{seccionppal} apply to show that $n_g=0$ for every $g\in G'$. The theorem follows then from Lemma \ref{lemaivanov}.

\end{proof}

\begin{obs}
There is a DR analogue of Theorem \ref{ivanovp} obtained by replacing the words ``aspherical'' by ``DR'' in the statement. The proof is identical, considering a finite subcomplex $L$ of $\widetilde{K}_{\mathcal{Q}'}$ without free faces, then proving that it cannot have cells of type $e^2_{m+1,g}$ and then using the hypothesis on $\p$ to conclude that $L$ is $1$-dimensional.

The DR versions of Proposition \ref{upp} and Theorem \ref{ivanovli} do not hold. In fact $\p=\langle | \rangle$, $\mathcal{Q}=\langle x| x^2x^{-1} \rangle$ show that even when $\p$ is DR, $\mathcal{Q}$ can be non-DR. 
\end{obs}

\section{An extension of the test} \label{seccionextended}

The ideas of the previous sections can be applied to a obtain a more general method. We use this extended test to prove diagrammatic reducibility of a larger class of LOTs and to study asphericity of presentations when our original I-test fails.

\begin{defi} \label{defposettest}
A function $\varphi: G\to P$ from a group $G$ to a poset $P$ is called \textit{order preserving} if $\varphi (h)\le \varphi (h')$ implies $\varphi (gh)\le \varphi (gh')$ for each $g\in G$ (in particular $\varphi (h)< \varphi (h')$ implies $\varphi (gh)< \varphi (gh')$). More generally, a function $\varphi : \coprod\limits_{j=1}^m G=\bigcup\limits_{j=1}^m G\times \{j\}\to P$ is \textit{order preserving} if $\varphi(h,j)\le \varphi(h',j')$ implies $\varphi(gh,j)\le \varphi(gh',j')$ for every $g\in G$. 

We say that a presentation $\p=\langle x_1, x_2, \ldots , x_n | r_1,r_2,\ldots , r_m \rangle$ of a group $G$ satisfies the \textit{extended test} if there exists a poset $P$ and an order preserving map $\varphi :\coprod\limits_{j=1}^m G\to P$ such that for every $1\le j\le m$ there is some $1\le i\le n$ with the following property: there exists $k\in \occ (x_i,r_j)$ such that $\varphi (s(k,r_j),j)>\varphi (s(k',r_{j'}),j')$ for every $1\le j'\le m$ and $k'\in \occ (x_i,r_{j'})$ with $(k',j')\neq (k,j)$.
\end{defi}

\begin{prop}
If $\p$ satisfies the extended test, it is DR. 
\end{prop}
\begin{proof}
Let $\{n^j_g\}_{j,g}$ be a nontrivial family of integer numbers with finite support such that each equation $E_{i,g}$ has zero or at least two nonzero terms. Let $P$ and $\varphi$ be as in the definition. The finite subposet $\{\varphi (g,j) | n^j_g\neq 0\}\subseteq P$ has at least one minimal element, say $\varphi(h,j)$ for $h\in G, 1\le j\le m$. Let $1\le i\le n$ and $k\in \occ(x_i,r_j)$ be as in Definition \ref{defposettest}. Let $g=hs(k,r_j)^{-1}$. Since $\varphi(gs(k',r_{j'}),j')<\varphi(gs(k,r_j),j)=\varphi(h,j)$ for each $k'\in \occ(x_i,r_{j'})$ with $(k',j')\neq (k,j)$, the equation $E_{i,g}$ has a unique nonzero term, which is a contradiction. 
\end{proof}

\begin{ej}
The I-test is a very particular case of the extended test. Suppose $\p$ satisfies the I-test for $v\in \R^n$ and the orders $j_1,j_2,\ldots ,j_m$ and $i_1, i_2,\ldots ,i_m$ of columns and rows. Consider $P=\R \times \{j_1<j_2<\ldots <j_m\}$ with the lexicographic order and let $\varphi: \coprod G \to P$ be defined by $\varphi (g,j)=(\overline{g},j)$. Then $\varphi$ is order preserving and the fact that the weight matrix $M$ is good implies that $\varphi$ satisfies the required property.
\end{ej}

Suppose now that $\p=\langle x_1,x_2,\ldots ,x_n| r_1,r_2,\ldots ,r_m \rangle$ is a presentation and let $1\le n_1<n_2<\ldots <n_k=n$, $1\le m_1<m_2<\ldots <m_k=m$ be such that for each $1\le l\le k$, the relators $r_1,r_2,\ldots , r_{m_l}$ are words in the first $n_l$ generators. In particular, if $v\in \R^{n}$ is orthogonal to each $q(r_j)$, the weight matrix $M(v)$ is a block matrix

\begin{displaymath}
\left(\begin{array}{c|c|c|c}
M_1(v) & * & \cdots & * \\
\hline
\emptyset & M_2(v) &  & * \\
\hline
\vdots &  & \ddots & \vdots \\
\hline
\emptyset & \emptyset & \cdots & M_k(v) \\
\end{array}\right)
\end{displaymath}

Here $M_l(v)$ is the submatrix of $M(v)$ corresponding to rows $n_{l-1}+1, n_{l-1}+2,\ldots, n_l$ and columns $m_{l-1}+1, m_{l-1}+2,\ldots, m_l$.

\begin{teo} \label{bloques}
Let $\p=\langle x_1,x_2,\ldots ,x_n| r_1,r_2,\ldots ,r_m \rangle$ and let $1\le n_1<n_2<\ldots <n_k=n$, $1\le m_1<m_2<\ldots <m_k=m$ be such that for each $1\le l\le k$, the relators $r_1,r_2,\ldots , r_{m_l}$ are words in the first $n_l$ generators. Let $v_1,v_2,\ldots,v_k\in \R^{n}$ such that each $v_l$ is orthogonal to every $q(r_j)$ and for each $1\le l\le k$, $M_l(v_l)$ is good (see Definition \ref{defiasph}). Then $\p$ is DR.
\end{teo}
\begin{proof}
We show that $\p$ satisfies the extended test. Let $j_{m_{l-1}+1}, j_{m_{l-1}+2}, \ldots , j_{m_{l}}$ be an order of the columns $m_{l-1}+1, m_{l-1}+2, \ldots , m_l$ that makes $M_l(v_l)$ good. Recall that for posets $P, Q$, the join $P*Q$ is the order in the disjoint union of $P$ and $Q$ obtained by preserving the ordering in each copy and setting $p<q$ for each $p\in P$ and $q\in Q$. Define then $$P= (\R \times \{j_{m_{k-1}+1}<j_{m_{k-1}+2}< \ldots < j_{m_{k}}\}) * (\R \times \{j_{m_{k-2}+1}< j_{m_{k-2}+2}< \ldots < j_{m_{k-1}} \}) * \ldots$$ $$\ldots * (\R \times \{j_{1}< j_{2}< \ldots <j_{m_{1}}\}),$$ in which each factor $\R \times \{j_{m_{l-1}+1}, j_{m_{l-1}+2}, \ldots , j_{m_{l}}\}$ is considered with the lexicographic order. Define $\varphi: \coprod\limits_{j=1}^m G\to P$ by $\varphi (g,j)=(\overline{g},j)$, where the weight $\overline{g}$ is taken with respect to the vector $v_l$ if $m_{l-1}+1\le j\le m_l$. Then $\varphi$ is order preserving and $\p$ satisfies the extended test.
\end{proof}

\begin{obs}
Note that even if all the vectors $v_l$ in the previous theorem are equal, it may happen that the blocks $M_l(v_l)$ are good while the weight matrix $M$ is not.
\end{obs}

Theorem \ref{bloques} can be used to generalize the notion of deforestability for LOGs.

\begin{defi} \label{wd}
A LOG $\Gamma$ is \textit{weakly deforestable} if there is a sequence $\Gamma=\Gamma_0\supset \Gamma_1 \supset \ldots \supset \Gamma_k$ of sub-LOGs of $\Gamma$ with $\Gamma_k$ discrete, such that for each $1\le i\le k$ there is a deforestation of type IL/T (see Definition \ref{defo}) or a deforestation of type TL/I from $\Gamma_{i-1}$ to $\Gamma_{i}$.  
\end{defi}

\begin{ej}

Figure \ref{cruz} shows a non-deforestable LOT $\Gamma$ of diameter $4$. There is a deforestation of type IL/T from $\Gamma$ to the sub-LOT $\Gamma '$ given by the horizontal edges. This sub-LOT is deforestable. Hence $\Gamma$ is weakly deforestable.

\begin{figure}[h] 
\begin{center}
\includegraphics[scale=0.32]{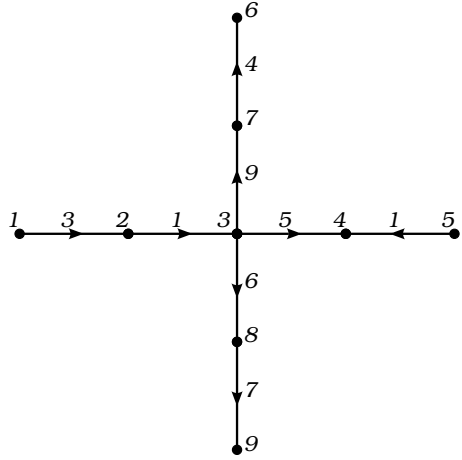}

\caption{A weakly deforestable LOT of diameter $4$.}\label{cruz}
\end{center}
\end{figure}

\end{ej}

\begin{teo}
If a LOG $\Gamma$ is weakly deforestable, the associated presentation $\p$ satisfies the hypothesis of Theorem \ref{bloques}. In particular $K_\p$ is DR.
\end{teo}
\begin{proof}
If $\Gamma=\Gamma_0\supset \Gamma_1 \supset \ldots \supset \Gamma_k$ are as in Definition \ref{wd}, there exist $1\le n_1<n_2<\ldots <n_k=n$, $1\le m_1<m_2<\ldots <m_k=m$ such that $\p=\p _{\Gamma}=\langle x_1,x_2,\ldots ,x_n| r_1,r_2,$ $\ldots ,r_m \rangle$ and for each $1\le l\le k$, the relators $r_1,r_2,\ldots , r_{m_l}$ are words in the first $n_l$ generators. Moreover, $n_l-n_{l-1}=m_{l}-m_{l-1}$ is the number of vertices (and edges) in $\Gamma_{k-l} \smallsetminus \Gamma_{k-l+1}$ for each $l\ge 2$. There exist $v_1, v_2, \ldots , v_k \in \R^n$, each $v_l$ equal to $(1,1,\ldots, 1)$ or to $(-1,-1,\ldots, -1)$, such that $M_1(v_1)$ and the square matrices $M_2(v_2), \ldots, M_k(v_k)$ are good. Then Theorem \ref{bloques} applies. 
\end{proof}

A particular case of the extended test which gives a useful generalization of the I-test is the following. Suppose $L$ is a left-orderable group and $\psi:G\to L$ is a group homomorphism. Then for an ordering $j_1,j_2,\ldots ,j_m$ of $\{1,2,\ldots , m\}$, consider $P=L\times \{j_1<j_2<\ldots <j_m\}$, again with the lexicographic order. The map $\varphi :\coprod\limits_{j=1}^m G\to P$ given by $\varphi (g,j)=(\psi(g),j)$ is order preserving and can be used to test the diagrammatic reducibility of $\p$.

\begin{ej}
Let $\p=\langle x,y,z | xzx^{-1}yzy^{-1}z^{-1}y^{-1}, xy^{-1}x^{-1}y^{-1}xy \rangle$ be a presentation of a group $G$. The subspace of vectors orthogonal to $q(r_1)$ and $q(r_2)$ is $\langle (1,1,1) \rangle \subseteq \R^3$. It is easy to check that for any $v\in \langle (1,1,1) \rangle$, the weight matrix $M(v)$ is not good. However we will see that the extended test can be applied to prove diagrammatic reducibility of $\p$.

Let $B_4=\langle x,y,z | xzx^{-1}z^{-1}, xyxy^{-1}x^{-1}y^{-1}, yzyz^{-1}y^{-1}z^{-1}\rangle$ be the braid group of 4-braids. The (opposite of the) Dehornoy ordering is a left-ordering in $B_4$ which satisfies the following: if an element $g\in B_4$ is represented by a word $w$ in the generators which contains the letter $x$ and no occurrence of $x$ in $w$ has positive exponent, then $g>1$ in $B_4$ (see \cite[Section 7.2]{CR} for more details). Let $\varphi: G\coprod G \to B_4$ be the map whose restriction to each copy of $G$ is the natural homomorphism $G\to B_4$.

The $3\times 2$ matrix $M$ defined by $M_{i,j}=\{\varphi(s(k,r_j),j)\}_{k\in \occ(x_i,r_j)}$ is 

\begin{displaymath}\bordermatrix{&\phantom{-}r_1& & \phantom{-}r_2 \cr 
        x& \phantom{-} 1,z^{-1} &  & \phantom{-}1,y^{-1}xy,xy   \cr
        y& z^{-1},z^{-1}y^{-1},1 &  & yx^{-1}, xy,y  \cr 
        z& x^{-1},y^{-1}z^{-1},y^{-1} &  & \phantom{-} \emptyset \cr}
\end{displaymath}

Then $x^{-1}$ is the greatest element of the third row while $yx^{-1}$ is the greatest of the second one. Therefore $\p$ satisfies the extended test.

\end{ej}


\section{Applications to equations over groups}

Recall that a system $S$ of equations over a group $H$ with unknowns $x_1,x_2, \ldots ,x_n$ is a set $\{w_j(x_1,x_2,\ldots, x_n)\}_j$ of words in $H* F(x_1,\ldots,x_n)$ . The letters of $w_j$ which lie in $H$ are called the coefficients of $w_j$. The non-necessarily reduced word $r_j$ in the alphabet $\{x_1, x_1^{-1}, x_2,x_2^{-1}, \ldots ,x_n,x_n^{-1}\}$  which is obtained by deleting the coefficients of $w_j$ will be called the \textit{shape} of $w_j$, and the word $r_j$ considered as an element of the free group $F(x_1,x_2,\ldots, x_n)$ will be called the \textit{content} of $w_j$. Note that the content of $w_j$ is just the image of $w_j$ under the canonical map $H*F(x_1,\ldots,x_n)\to F(x_1,\ldots,x_n)$ which maps $H$ to the identity.
 
We say that the system $S$ has a solution in an overgroup of $H$ if there exits a group $H'$ which contains $H$ as a subgroup and elements $h_1, h_2, \ldots ,h_n$ in $H'$ such that $$w_j(h_1,h_2,\ldots, h_n)=1\in H'$$ for every $j$. The Kervaire-Laudenbach Conjecture states that for any group $H$, a unique equation $w$ with a unique unknown $x$ has a solution in an overgroup of $H$  if $w$ is non-singular (i.e. if the total exponent of $x$ in $w$ is non-zero). This conjecture has been proved in many cases, for different groups $H$ and/or equations $w$. The so called Kervaire-Laudenbach-Howie Conjecture generalizes this to an arbitrary finite number $n$ of unknowns and a non-singular system of $m$ equations (in this case, non-singular means that the $m\times n$ matrix of total exponents has rank equal to $m$). This generalized conjecture has also been verified in various cases. For example, Howie proved that it holds for locally indicable groups \cite{Ho1}.

Let $S$ be a system of equations $w_1,w_2,\ldots ,w_m$ over a group $H$. Let $\p$ be the presentation $\langle x_1,x_2, \ldots , x_n | r_1,r_2, \ldots ,r_m \rangle$ whose generators are the unknowns of $S$ and its relators are the shapes of the equations $w_j$. A well known result by Gersten \cite{Ger} states that if $\p$ is DR, then $S$ has a solution in an overgroup of $H$. In other words, for any group $H$, any system of equations modeled by the presentation $\p$ has a solution in an overgroup of $H$. A presentation with this property is said to be \textit{Kervaire}. The converse of this result is false. The presentation $\p=\langle t | ttt^{-1}\rangle$ is not DR, but it is Kervaire: any equation $atbtct^{-1}$ modeled by $\p$ over any coefficient group $H$ has a solution in an overgroup of $H$ \cite{Hol3}.

We concentrate now on solutions of one equation $w$ with many unknowns. By a result attributed to Pride, for any coefficient group $H$, if the shape of $w$ is cyclically reduced (and non-trivial) then the equation has a solution in an overgroup of $H$ (see \cite[Corollary 5.7]{Ger}). In the same direction, as a consequence of a result of Brick  \cite[Proposition 4.1]{br2}, for an arbitrary group $H$, if the shape of $w$ is reduced (although not necessarily cyclically) and it is not a proper power, then there is a solution in an overgroup of $H$.


Recently Klyachko and Thom \cite{kt} proved the following result for one equation with many variables over hyperlinear groups. Note that their result does not depend on the shape of the equation but only on its content.


\begin{teo} [Klyachko - Thom] \label{ktteo}
Let $G$ be a hyperlinear group. An equation in two variables with coefficients
in $G$ can be solved over $G$ if its content does not lie in $[F_2, [F_2,F_2]]$. Moreover, if $G$ is finite,
then a solution can be found in a finite extension of $G$.
\end{teo}

Here $F_2$ denotes the free group generated by the two unknowns. In general we denote by $F_n$ the free group generated by the variables $x_1,\ldots,x_n$.

We use our methods to prove Kervaireness in many cases which are not covered by Theorem \ref{ktteo} and the previous results of Brick \cite{br2}.

\begin{lema} \label{capsula}
Let $\p=\langle x_1,x_2, \ldots, x_n | r\rangle$ be a one-relator presentation with $r\in [F_n, F_n]$. For each $k\in \occ(x_i,r)$, let $v_k$ denote the class of $s(k,r)\in F_n$ in $F_n^{ab}=\Z^n$. Let $P_i$ denote the convex hull of the set $\{v_k\}_{k\in \occ(x_i,r)}$. If for some $1\le i\le n$ there exists $k\in \occ(x_i,r)$ such that $v_k$ is a vertex of the polytope $P_i$ and the multiplicity of $v_k$ in $\{v_j\}_{j\in \occ(x_i,r)}$ is one, then $\p$ is DR, and therefore Kervaire.  
\end{lema}

\begin{proof}
Suppose $i$ and $k$ are in the hypothesis of the statement. Let $H$ by an hyperplane of $\R^n$ such that $H\cap P_i=\{v_k\}$. Let $v\in \R^n$ be a nonzero vector orthogonal to $H$. Then $\p$ satisfies the I-test with respect to $v$ or to $-v$.
\end{proof}

\begin{ej}
Let $\p=\langle x,y | [yx^{-1}, [x^2,y]] \rangle$. The generator $x$ occurs 10 times in the relator $yx^{-1}x^2yx^{-2}y^{-1}xy^{-1}yx^2y^{-1}x^{-2}$. The vectors $v_k$ for $k\in \occ(x,r)$ appear in Figure \ref{hull}

\begin{figure}[h] 
\begin{center}
\includegraphics[scale=1]{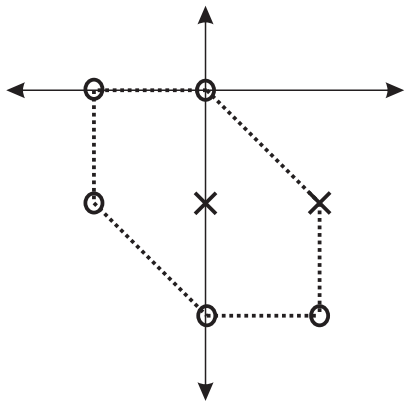}

\caption{The convex hull of the collection $\{v_k\}_{k\in \occ(x, r)}$.}\label{hull}
\end{center}
\end{figure}

The repeated vectors $(0,-1)$ and $(1,-1)$ are marked with a cross while the others appear with a circle. The vector $(-1,0)$ is a vertex of the convex hull and has multiplicity one. Therefore $\p$ is DR. 

\end{ej}

It is not difficult to see that Brick's result \cite[Proposition 3.1 (1)]{br2} can be derived from Lemma \ref{capsula}.

We now use the lemma to prove two results that cover many cases which are not contemplated in the results of Brick and Klyachko-Thom. Note that in both cases, the shapes of the equations involved are non-necessarily reduced.

\begin{prop}
Let $w_1, w_2, \ldots , w_n$ be nontrivial words in two variables $\{x,y\}$, each of them positive or negative, and let $w= [w_n,[w_{n-1}, \ldots , [w_2, [w_1, [x,y]]] \ldots ]]$. Then for any group $H$, any equation modeled by $w$ has a solution in an overgroup of $H$.
 
\end{prop}
\begin{proof}
We prove by induction that the presentation  $\p=\langle x,y | w \rangle $ satisfies the I-test with respect to the vector $v=(-1,-1)$. Moreover, we show that the collection $\{\overline{s(k,r)}\}_{k\in \occ(x, r)}$ has a positive maximum with multiplicity 1. For $n=0$, $x$ occurs twice in $r=xyx^{-1}y^{-1}$ and the collection is $\{0,1\}$ which satisfies our claim.

Assume now that $\mathcal{Q}=\langle x,y | u\rangle$ satisfies the claim for $u=[w_{n-1}, \ldots , [w_2, [w_1, [x,y]]] \ldots ]$. Concretely, the collection $\{\overline{s(k,u)}\}_{k\in \occ(x, u)}$ is an ordered list $a_1, a_2, \ldots , a_j$ where $j$ is the number of occurences of $x$ in $u$, and this list has a unique maximum, say $a_m$, and $a_m>0$. Then $x$ occurs $2j+2l$ times in $r=w_nuw_n^{-1}u^{-1}$ where $l$ is the number of occurences of $x$ in $w_n$. If $k\in \occ(x,w_n)$, then $k$ and $k'=2\len(w_n)+\len(w_n)-k+1$ are in $\occ(x,r)$. Moreover $\overline{s(k,r)}=\overline{s(k',r)}=\overline{s(k,w_n)}$ since the weight of $u$ is zero. If $k\in \occ(x,u)$, then $k'=\len(w_n)+k$ and $k''=2\len(w_n)+2\len(u)-k+1$ are in $\occ(x,r)$ and $\overline{s(k',r)}=\overline{s(k,u)}-\overline{w}_n$, while $\overline{s(k'',r)}=\overline{s(k,u)}$. In particular, the collection $\overline{s(k,r)}_{k\in \occ(x,r)}$ is an ordered list of the form $$b_1, b_2, \ldots, b_l, a_1-\overline{w}_n, a_2-\overline{w}_n, \ldots, a_j-\overline{w}_n, b_l, \ldots, b_2, b_1, a_j, \ldots , a_2, a_1.$$
In the case that $w_n$ is a positive word, the sequence $b_1, b_2, \ldots , b_l$ is strictly increasing, each $b_i$ is non-negative and $-\overline{w}_n> b_l$. In this case, $a_m-\overline{w}_n$ is the unique maximum of the list and it is positive.

In the case that $w_n$ is negative, $b_1, b_2, \ldots , b_l$ is strictly decreasing, each $b_i$ is negative and $-\overline{w}_n\le b_l$. In this case, $a_m$ is the unique maximum of the list and it is positive.
\end{proof}

If $w$ is a word in the alphabet $\{x,x^{-1}\}$ with  total exponent zero, then it is clear that there exists a cyclic permutation of $x$ which is a (non-negative) Dyck word, meaning that each proper initial segment of $w$ has non-negative total exponent. We will say that a word $w$ in $\{x,x^{-1}\}$ is a \textit{strong Dyck word} if each initial proper segment has positive total exponent (positive strong Dyck word) or if each initial proper segment has negative total exponent (negative strong Dyck word).  

If $w$ is a word in the alphabet $\{x,x^{-1},y,y^{-1}\}$, the \textit{$x$-shape} of $w$ is the word in $\{x,x^{-1}\}$ obtained by removing all the $y$ and $y^{-1}$ letters.

\begin{prop} \label{dyck}
Let $w\in [F_2,F_2]$ be a non-necessarily reduced word in two variables $x,y$. Suppose that $w$ has a cyclic permutation $w'=z_1^{l_1}z_2^{l_2}\ldots z_p^{l_p}$ (each $z_i\in \{x,y\}$, each $l_i$ a non-zero integer) with the following properties.

\noindent (a) The $x$-shape of $w'$ is a strong Dyck word.

\noindent (b) $z_1=z_{p-1}=x$, $z_p=y$.

Then any equation modeled by $w$ over any coefficient group $H$ has a solution in an overgroup of $H$.  
\end{prop}
\begin{proof}
We can assume that $w'=w$. We can also assume that the $x$-shape of $w$ is a positive strong Dyck word by replacing $x$ by $x^{-1}$ if necessary. Then the class $v_1$ of $s(1,w)=w\in F_2$ in $\Z^2$ is $(0,0)$. Assume that $k\in \occ(x,w)$ is the last occurence of $x$ in $w$. Then the class $v_k$ of $s(k,w)$ is $(0,l_p)\in \Z^2$. For any other occurence $k'\neq 1,k$ of $x$ in $w$, the first coordinate of $v_{k'}\in \Z^2$ is negative by hypothesis. Then $v_1$ (and $v_k$) is a vertex of the polytope $P_1$ with multiplicity one in $\{v_j\}_{j\in \occ(x,w)}$. By Lemma \ref{capsula}, $\langle x,y | w \rangle$ is DR.
\end{proof}

\begin{ej}
The presentation $\p=\langle x,y | [y^{-1}x,[x^{-1},y^{-2}x]]\rangle$ is Kervaire. We prove that $w=[y^{-1}x,[x^{-1},y^{-2}x]]$ satisfies the hypothesis of Proposition \ref{dyck}. The $x$-shape of $w=y^{-1}xx^{-1}y^{-2}x^2x^{-1}y^2x^{-1}yy^{-2}xx^{-2}y^2x$ is not a strong Dyck word. However, the $x$-shape of the cyclic permutation $w'=xy^{-1}xx^{-1}y^{-2}x^2x^{-1}y^2x^{-1}yy^{-2}xx^{-2}y^2$ of $w$ is $x^2x^{-1}x^2x^{-2}xx^{-2}$ which is a positive strong Dyck word. Moreover, following the notation of Proposition \ref{dyck} we have $p=14$, $z_1=z_{13}=x$ and $z_{14}=y$, so the result applies.  
\end{ej}

\end{document}